\documentclass[11pt, reqno]{amsart}[11pt]
\usepackage{geometry}                
\usepackage{graphicx}
\usepackage{epstopdf}
\usepackage{xcolor}
\usepackage{soul}
\newtheorem{theorem}{Theorem}[section]
\newtheorem{remark}{Remark}[section]

\usepackage[commandnameprefix=always]{changes}
\usepackage{amsmath,amssymb,amscd,mathrsfs,tikz-cd,mathtools,scalerel,latexsym}
\usepackage{bm}
\usepackage{color}
\numberwithin{equation}{section}
\DeclareMathOperator{\H2}{\mathcal{H}_{q^2, g/2+1}}
\DeclareMathOperator{\A2}{\mathcal{A}_{q^2}}
\DeclareMathOperator{\Hq}{\mathcal{H}_{q^2}}
\DeclareMathOperator{\Aq}{\mathcal{A}_{q}}
\DeclareMathOperator*{\Res}{Res}  
\newtheorem{lemma}{Lemma}[section]
\DeclareMathOperator{\moda}{mod}
\newtheorem{definition}{Definition}[section]

\title{mean value of cubic $L$-functions with fixed genus}

\author{Ziwei Hong}
\address{School of Mathematical Sciences, Renmin University of China, Beijing, P.R. China}
\email{hongziwei@live.com}

\author{Zhongqiu Fang}
\address{China Post Securities Co., Beijing, P.R. China}
\email{fangzhongqiu2024@outlook.com}

\date{April 2024}

\begin{document}

\begin{abstract}
We investigate the mean value of the first moment of primitive cubic $L$-functions over $\mathbb{F}_q(T)$ in the non-Kummer setting. Specifically, we study the sum
 \begin{equation*}
 \sum_{\substack{\chi\ primitive\ cubic\\ genus(\chi)=g}}L_q(\frac{1}{2}, \chi),
 \end{equation*}
where $L_q(s,\chi)$ denotes the $L$-function associated with primitive cubic character $\chi$. Employing a double Dirichlet series approach, we establish an error term of size $q^{(\frac{7}{8}+\varepsilon)g}$. Our method significantly reduces the computational complexity by avoiding computing the residue. 
\end{abstract}

\maketitle


\section{Introduction}
In this paper, we study the mean value of the first moment of primitive cubic Dirichlet $L$-functions $L_q(s,\chi)$ evaluated at the central point $s=\frac{1}{2}$, where $\chi$ ranges over primitive cubic characters of $\mathbb{F}_q[T]$ with fixed genus $g$ when $q \equiv 2 \moda 3$. Our main result confirms the asymptotic formula established by David, Florea, and Lalin \cite{2019The}, but our approach, inspired by the double Dirichlet series method of Gao and Zhao \cite{GAO2024125}, significantly reduces the computational complexity involved in deriving this result.

The first moment of cubic $L$-functions has been widely studied in the number field setting. Luo \cite{Luo_2004} analyzed the first moment of cubic Dirichlet twists over $\mathbb{Q}(\xi_3)$, while G\"ulo\u{g}lu and Yesilyurt \cite{gulouglu2024mollified} investigated the mollified first moment over the Eisenstein field. David et al. \cite{david2024non} studied the mollified second moment and obtained a power-saving error term over the Eisenstein field. Baier and Young \cite{BAIER2010879} considered the first moment of central values of Dirichlet $L$-functions associated with primitive cubic characters over $\mathbb{Q}$ and proved the asymptotic formula
\begin{align}\label{Baier}
    \sum_{(q,3)=1}\sum_{\substack{\chi~ primitive \moda q\\ \chi~cubic}} L(\frac{1}{2},\chi)w(\frac{q}{Q})=c\hat{w}(0)Q+O(Q^{\frac{37}{38}+\varepsilon}),
\end{align}
with an explicit constant $c$. This result was later refined by Gao and Zhao \cite{GAO2024125}, who employed double Dirichlet series to improve the error term to
\begin{align}\label{Gao}
     \sum_{(q,3)=1}\sum_{\substack{\chi~ primitive \moda q\\ \chi~cubic}} L(\frac{1}{2},\chi)\Phi(\frac{q}{Q})=C_3\hat{\Phi}(0)Q+O(Q^{\frac{7}{8}+\varepsilon}).
\end{align}

In the function field setting, the mean value of cubic $L$-functions was first explored by Rosen \cite{rosen15average}, who averaged over all monic polynomials of a given degree. A significant advancement was made by David, Florea, and Lalin \cite{2019The}, who investigated the mean value of the first moment of cubic $L$-functions $L_q(s,\chi)$ over primitive cubic characters of fixed genus when $q \equiv 2 \moda 3$, establishing the asymptotic formula
\begin{align}\label{F}
    \sum_{\substack{\chi~ primitive~cubic\\ genus(\chi)=g}}L_q(\frac{1}{2}, \chi)=\frac{\zeta_q(\frac{3}{2})}{\zeta_q(3)}\mathcal{A}_{nK}(\frac{1}{q^2},\frac{1}{q^{3/2}})q^{g+2}+O(q^{\frac{7}{8}g+\varepsilon g}).
\end{align}


Here $\zeta_q(s)=\frac{1}{1-q^{1-s}}$ and $\mathcal{A}_{nK}$ is an explicit function. 
Motivated by the work of Gao and Zhao \cite{GAO2024125}, we adapt their approach using double Dirichlet series to compute the first moment of cubic $L$-functions in the function field setting. Our main result is stated as follows

\begin{theorem} 
Let $q$ be an odd prime power such that $q\equiv 2\moda 3$. Then
    \begin{align}\label{main}
     \sum_{\substack{\chi\ primitive\ cubic\\ genus(\chi)=g}}L_q(\frac{1}{2}, \chi)=q^g(1-q^2)P(q^{-2})Z(q^{-2},q^{-\frac{1}{2}}) +O(q^{(\frac{7}{8}+\varepsilon)g}),
    \end{align}
    with $P(u)$ and $Z(u,v)$ given in Section~\ref{PZ}.
\end{theorem}

After some simplification, it's easy to ses that our result confirms the asymptotic formula established by David, Florea, and Lalin \cite{2019The}, but our approach significantly streamlines the computation. Moreover, in \cite{2019The}, the authors identified a potential secondary main term arising from the residue of a generating function for cubic Gauss sums. However, this term ultimately did not appear in their final asymptotic formula. We will clarify why this term contributes nothing to the main term in the asymptotic expansion.

Specifically, in \cite{2019The}, considerable effort was devoted to computing the residue of the generating function at a pole that would, in principle, produce a term of size $q^{\frac{5g}{6}}$. However, the dominant error term originates from the integral along the circle $|u| = q^{-\frac{7}{8}+\varepsilon}$, contributing an error of size $q^{(\frac{7}{8}+\varepsilon)g}$. As a result, the supposed secondary main term is overshadowed by the error term and does not appear in the final asymptotic formula.

In this paper, we adopt a different method that not only simplifies the overall computation but also predicts in advance that this residue lies outside the relevant region of convergence. We explicitly show that the corresponding pole lies beyond the convex hull of the region of absolute convergence, rendering its contribution negligible.

David, Florea and Lalin's method relies on an approximation function to express the sum as a collection of short sums, which are then evaluated using Perron’s formula. Each of these short sums contributes a main term and an error term. In contrast, our method eliminates the need for an approximation function. We rewrite the sum in a form that directly reveals the convex hull of the region of convergence. By applying Perron’s formula to this reformulated expression, we obtain the main term from the residues and estimate the error term through boundary integration. This streamlined approach reduces the complexity of the analysis while maintaining the strength of the result.



\section{Preliminary}
In this section, we recall some standard properties of Gauss sums and the functional equation for $L$-functions. Since the proofs of these results are well known and readily available in the literature, we omit them here. Readers are encouraged to focus on the quantitative relationships established by these properties, as they play a crucial role in the subsequent computational analysis.

We work in the non-Kummer setting, where q is an odd prime power such that $q \equiv 2 \moda{3}$. For simplicity, we do not consider the corresponding results in the Kummer setting ($q \equiv 1 \moda{3}$), although a similar asymptotic formula for the first moment of cubic $L$-functions in that case can be derived through analogous methods.


\subsection{Primitive cubic characters and $L$-functions}
In the non-Kummer setting, defining primitive cubic characters over $\mathbb{F}_q[T]$ is more intricate. Following the approach of \cite{2017On}, it is natural to restrict a cubic character defined over $\mathbb{F}_{q^2}[T]$ to obtain a cubic character over $\mathbb{F}_q[T]$.

We fix once and for all an isomorphism $\Omega$ between the cubic roots of $1$ in $\mathbb{C}^*$ and the cubic roots of $1$ in $\mathbb{F}^*_{q^2}$. We define the cubic residue symbol $\chi_{\pi}$, for $\pi$ an irreducible monic polynomial in $\mathbb{F}_{q^2}[T]$ such that $\pi\pi^{\sigma}=P$ is an irreducible polynomial in $\mathbb{F}_q[T]$, where $\sigma$ is the generator of $Gal(\mathbb{F}_{q^2}/\mathbb{F}_q)$. Let $a\in\mathbb{F}_{q^2}[T]$. If $\pi |a$, then $\chi_{\pi}(a)=0$, and otherwise $\chi_{\pi}(a)=\alpha$, where $\alpha$ is the unique root of unity in $\mathbb{C}$ such that 
\begin{equation*}
    a^{\frac{q^{2\deg\pi}-1}{3}}\equiv\Omega(\alpha)\moda \pi.
\end{equation*}
Then $\chi_{\pi}|_{\mathbb{F}_q[T]}$ is a cubic character module $P$.

We extend the definition by multiplicity. For any monic polynomial $F\in\mathbb{F}_{q^2}[T]$, $F=\pi_1^{e_1}\dots\pi_s^{e_s}$ with distinct $\pi_i$, we define $\chi_F=\chi_{\pi_1}^{e_1}\dots\chi_{\pi_s}^{e_s}$ . We restrict $\chi_{F}$ to $\mathbb{F}_q[T]$, then $\chi_{F}$ is a cubic character on $\mathbb{F}_q[T]$. $\chi_F$ is primitive if and only if $e_i=1$ for $i=1,2,...,s$ and $F$ has no divisor in $\mathbb{F}_q[T]$. It it is so, then $\chi_F$ has conductor $FF^{\sigma}$.

 Let $\mathcal{A}_q$ denote the set of monic polynomials over $\mathbb{F}_q$. Let $\H2$ represent the set of monic, square-free polynomials in $\A2$ of degree $g/2+1$. With above notations, in \cite{2019The}, David et al show that
 \begin{lemma}[Lemma~2.11 in \cite{2019The}]
      Suppose $q\equiv 2\moda 3$, Then,
     \begin{equation}\label{equation}
 \sum_{\substack{\chi\ primitive\ cubic\\ genus(\chi)=g}}L_q(\frac{1}{2}, \chi)=  \sum_{\substack{F\in\H2 \\ P|F\Rightarrow P\not\in \mathbb{F}_q[t]}}L_q(\frac{1}{2}, \chi_{F}).
 \end{equation}
 \end{lemma}
 As 
 \begin{equation*}
     \chi_F(\alpha)=\Omega^{-1}\left(\alpha^{\frac{q^{2\deg F}-1}{3}}\right)
 \end{equation*}
 for $\alpha\in\mathbb{F}_q\subset\mathbb{F}_{q^2}$, and $q$ is odd and $q\equiv 2\moda 3$, we remark that all cubic characters over $\mathbb{F}_q[T]$ are even. Hence, we have the following functional equation:
 \begin{lemma}[Functional equation for even $L$-functions] \label{fe}
     Let $F\in\mathcal{H}_{q^2}$ and suppose $F$ has no divisor in $\Aq$. Then
     $$L_q(s,\chi_{F})=\epsilon(\chi_F)q^{2s-1}\frac{1-q^{-s}}{1-q^{s-1}}\frac{L_q(1-s,\overline{\chi_F})}{|F|_2^{s-\frac{1}{2}}},$$
     where $\epsilon(\chi_F)=q^{-\deg F}G(\chi_F)$, $G(\chi_F)$ is the Gauss sum.
 \end{lemma}
 Here $|F|_2=q^{2\deg F}$ stands for the modulus of polynomials in $\mathbb{F}_{q^2}[T]$.
 Next, we present the upper bound of $L$-functions. The following lemma restates Lemma 2.6 and Lemma 2.7 from \cite{2019The}.
 \begin{lemma}[Lindel\"of Hypothesis]
 \label{LLH}
     Let $\chi$ be a primitive cubic character of conductor $h$ defined over $\mathbb{F}_q[T]$. Then, for $\Re(s)\ge\frac{1}{2}$ and all $\varepsilon>0$,
     \begin{align}
         |L_q(s,\chi)|\ll q^{\varepsilon\deg h};
     \end{align}
     for $\Re(s)\ge 1$ and for all $\varepsilon>0$
     \begin{align}
         |L_q(s,\chi)|\gg q^{-\varepsilon\deg h}.
     \end{align}
 \end{lemma}


\subsection{Gauss sums}
 
We recall some well-know properties of polynomial Gauss sums. For details about polynomial Gauss sums, see \cite{ZHENG2017460}.
 
 \begin{definition}
     For $\chi$ a primitive character of the modulus h on $\mathbb{F}_q[T]$, let
     $$G_q(\chi)=\sum_{a \moda h}\chi(a)e_q(\frac{a}{h}).$$
 \end{definition}
 We also define generalized Gauss sum as follow.
 \begin{definition}
 Let $\chi_f$ be character of $\mathbb{F}_q[T]$ for some $f\in\mathbb{F}_q[T]$. We define generalized Gauss sum
     \begin{align}
          G_q(V, f)=\sum_{u\moda f}\chi_f(u)e_q\left(\frac{uV}{f}\right).
     \end{align}
 \end{definition}

 \begin{lemma}
\label{Gauss1}
    Suppose that $q\equiv 1 (\moda 6)$ and $f_1$, $f_2$ are arbitrary polynomials in $\Aq$.
    \begin{enumerate}
        \item If $(f_1,f_2)=1$, then
        \begin{align*}
            G_q(V,f_1f_2)=&\chi_{f_1}(f_2)^2G_q(V,f_1)G_q(V,f_2)\\
            =&G_q(Vf_2,f_1)G_q(V,f_2).
        \end{align*}
        \item If $V=V_1P^{\alpha}$ where $P\not| V_1$, then
        \begin{equation}
            G_q(V,P^i)=\left\{
            \begin{tabular}{ll}
                $0$, & if $i\le\alpha$ and $i\not\equiv 0 \moda 3$;\\
                $\phi(P^i)$, & if $i\le\alpha$ and $i\equiv 0 \moda 3$;\\  
                $-|P|^{i-1}$, & if $i=\alpha+1$ and $i\equiv 0\moda 3$;\\              $\epsilon(\chi_{P^i})\chi_{P^i}(V_1^{-1})|P|^{i-\frac{1}{2}}$, & if $i=\alpha+1$ and $i\not\equiv 0\moda 3$;\\
                $0$, & if $i\ge\alpha+2$,            \end{tabular}\right.
        \end{equation}
        where $\phi$ is the Euler function for polynomials. We recall that $\epsilon(\chi)=1$ when $\chi$ is even. For the case of $\chi_{P^i}$, this happens if $3|\deg P^i$.
    \end{enumerate}
\end{lemma}
 This is Lemma~2.12 from \cite{2019The}.
 For the described primitive character $\chi_F$, we have
 \begin{lemma}\label{Gauss2}
 For $F\in\mathbb{F}_{q^2}[T]$, we have
     $$G(\chi_F)=G_{q^2}(1,F).$$
     Moreover, if $F$ is square-free, then
          \begin{align}
         |G_{q^2}(1,F)|=q^{\deg F}.
     \end{align}
 \end{lemma}


\subsection{Distribution of Gauss sums}

The generating series for Gauss sums is defined as
\begin{align}
    \Psi_q(f, u)=\sum_{F\in\Aq}G_q(f,F)u^{\deg F}
\end{align}
and
\begin{align}  \widetilde{\Psi}_q(f,u)=\sum_{\substack{F\in\Aq\\ (F,f)=1}}G_q(f,F)u^{\deg F}
\end{align}

We now discuss the distribution of Gauss sums. Patterson shows the distribution of cubic Gauss sums for $S$-integers (see \cite{https://doi.org/10.1112/plms/s3-54.2.193}). With some minor adjustments, we have the following results.

\begin{lemma}[Lemma~3.11 in \cite{2019The}]
\label{upper2}
    Let $f=f_1f_2^2f_3^3$ with $f_1$, $f_2$ square-free and co-prime, and $f_3^*$ be the product of the primes dividing $f_3$ but not dividing $f_1f_2$. Then,
    \begin{align*}
        \widetilde{\Psi}_q(f,u)=\prod_{P|f_1f_2}(1-(u^3q^2)^{\deg P})^{-1}\sum_{a|f_3^*}\mu(a)G_q(f_1f_2^2,a)u^{\deg a}\prod_{P|a}(1-(u^3q^2)^{\deg P})^{-1}\\
        \times\sum_{l|af_1}\mu(l)(u^2q)^{\deg l}\overline{G_q(1,l)}\chi_l(af_1f_2^2/l,u)\Psi_q(af_1f_2/l,u).
    \end{align*}
    If $\frac{1}{2}\le \sigma\le \frac{3}{2}$ and $|u^3-q^{-4}|$, $|u^3-q^{-2}|>\delta$, then
    $$\widetilde{\Psi}_q(f,u)\ll |f|^{\frac{1}{2}(\frac{3}{2}-\sigma)+\varepsilon}.$$
\end{lemma}

\subsection{Multivarible Complex Analysis Lemmas}
Our approach relies on two foundational results from multivariable complex analysis. We begin by introducing the concept of a tube domain.
\begin{definition}
		An open set $T\subset\mathbb{C}^n$ is a tube if there is an open set $U\subset\mathbb{R}^n$ such that $T=\{z\in\mathbb{C}^n:\ \Re(z)\in U\}.$
\end{definition}
	
   For any set $U\subset\mathbb{C}^n$, we define $T(U)=U+i\mathbb{R}^n\subset \mathbb{C}^n$.  We quote the following Bochner's Tube Theorem \cite{Boc}.
\begin{theorem}
\label{Bochner}
		Let $U\subset\mathbb{R}^n$ be a connected open set and $f(z)$ be a function holomorphic on $T(U)$. Then $f(z)$ has a holomorphic continuation to
the convex hull of $T(U)$.
\end{theorem}

 We denote the convex hull of an open set $T\subset\mathbb{C}^n$ by $\widehat T$.  Our next result is \cite[Proposition C.5]{Cech1} on the modulus of holomorphic continuations of multivariable complex functions.
\begin{theorem}

\label{Extending inequalities}
		Assume that $T\subset \mathbb{C}^n$ is a tube domain, $g,h:T\rightarrow \mathbb{C}$ are holomorphic functions, and let $\tilde g,\tilde h$ be their
holomorphic continuations to $\widehat T$. If  $|g(z)|\leq |h(z)|$ for all $z\in T$ and $h(z)$ is nonzero in $T$, then also $|\tilde g(z)|\leq
|\tilde h(z)|$ for all $z\in \widehat T$.
\end{theorem}

\section{Analytical behavior of $A_3(u,v)$}
  \subsection{The generating function}
 Now, we construct a double Dirichlet series to rewrite the sum in equation \eqref{equation}. From now on, we will always use $P$, $P_1$ or $P_2$ to denote irreducible polynomial in corresponding polynomial ring. For $\Re(s)$ and $\Re(w)$ sufficiently large, we define
 \begin{equation}\label{A3}
     A_3(s, w)=\sum_{\substack{F\in\mathcal{H}_{q^2} \\ P|F\Rightarrow P\not\in \mathbb{F}_q[T]}} \frac{L_q(w, \chi_{F})}{|F|^s_2}.
 \end{equation}
 By Perron's formula for function fields, we have
 \begin{equation}
     \sum_{n\le N}a(n)=\frac{1}{2\pi i}\oint_{|u|=r}(\sum_{n=0}^{\infty}a(n)u^n)\frac{du}{(1-u)u^{N+1}},
 \end{equation}
which allows us to rewrite the sum as 
\begin{align}
    \sum_{\substack{F\in\H2 \\ P|F\Rightarrow P\not\in \mathbb{F}_q[t]}}L_q(\frac{1}{2}, \chi_{F})
    =& \frac{1}{2\pi i}\oint_{|u|=r}\sum_{\substack{F\in\mathcal{H}_{q^2} \\ P|F\Rightarrow P\not\in \mathbb{F}_q[t]}} L_q(\frac{1}{2}, \chi_{F}) u^{\deg F}\frac{du}{u^{g/2+2}}\notag\\
    =&\frac{1}{2\pi i}\oint_{|u|=r}A_3(u,\frac{1}{2})\frac{du}{u^{g/2+2}},
\end{align}
where $A_3(u, w)=A_3(s, w)$ upon taking $u=q^{-2s}$. Hereafter, we will use $A_3(s,w)$ and $A_3(u,w)$ interchangeably when there is no ambiguity.

Next, we consider the different expressions of $A_3(s,w)$ on different convergence regions and find out the convex hull of convergence regions.

 \subsection{The first convergence region}
 To analyze the behavior of $A_3(u,v)$, we begin by using M\"obius inversion 
 \begin{equation}     
 \sum_{\substack{D\in\Aq\\ D|F}}\mu(D)=\left\{ \begin{array}{ll}
         1 & \mbox{$F$ has no prime divisor in } \mathbb{F}_q[T]; \\
         0 & \mbox{otherwise}.
     \end{array}
     \right.
 \end{equation}
 This allows us to remove the divisor condition.
 Applying this, $A_3(s,w)$ can be rewritten as
  \begin{align}
  A_3(s, w)=&\sum_{\substack{F\in\Hq \\ P|F\to P\not\in \mathbb{F}_q[T]}}\frac{L_q(w,   \chi_{F})}{|F|_2^s}\notag\\
  =&\sum_{D\in\Aq} \mu(D)\sum_{\substack{F\in\mathcal{H}_{q^2}\\ (D,F)=1}} \frac{L_q(w, \chi_{DF})}{|DF|_2^s}\notag\\
  =&\sum_{D\in\Aq}\mu(D)\sum_{\substack{F\in\mathcal{H}_{q^2}\\ (D,F)=1}}\sum_{N\in\Aq}\frac{\chi_{DF}(N)}{|N|^{w}|DF|_2^s}\notag\\
  =&\sum_{N\in\Aq}\frac{1}{|N|^w}\sum_{\substack{D\in\Aq\\ }}\frac{\mu(D)\chi_D(N)}{|D|_2^s}\sum_{\substack{F\in\Hq\\ (D,F)=1}}\frac{\chi_F(N)}{|F|_2^s}\notag\\
  =&\sum_{N\in\Aq}\frac{1}{|N|^w}\frac{L_{q^2}(s, \chi^{(N)})}{L_{q^2}(2s, \Bar{\chi}^{(N)})}\sum_{\substack{D\in\Aq\\ }}\frac{\mu(D)\chi_D(N)}{|D|_2^s}\prod_{\substack{P_2|D\\ P_2\in\A2}}\left(1+\frac{\chi_{P_2}(N)}{|P_2|_2^s}\right)^{-1}.
 \end{align}
 The last equality uses a trick
 \begin{align*}
     \sum_{\substack{F\in\Hq\\ (D,F)=1}}\frac{\chi_F(N)}{|F|_2^s}=\prod_{\substack{P_2\in\A2\\ (P_2,D)=1}}\left(1+\frac{\chi_{P_2(N)}}{|P_2|^s_2}\right)=\frac{L_{q^2}(s,\chi^{(N)})}{L_{q^2}(2s,\overline{\chi^{(N)})}}\prod_{\substack{P_2\in\A2\\ P_2|D}}\left(1+\frac{\chi_{P_2}(N)}{|P_2|_2^s}\right)^{-1}. 
 \end{align*}
 Here $\chi^{(N)}(F)=\chi_F(N)$ is the Hecke character induced by $\chi_F$. $L_{q^2}(s,\chi^{(N)})$ denotes the $L$-function over $\mathbb{F}_{q^2}(T)$. Meanwhile, we recall that $|F|=q^{\deg F}$ and $|F|_2=q^{2\deg F}$.

 Using the lemma below, we simplify $A_3(s,w)$ further.
 \begin{lemma}
     If $D, F\in\Aq$ and $(D, F)=1$, then we have $\chi_D(F)=1$.
 \end{lemma}
 \begin{proof}
     Let $\sigma: \mathbb{F}_{q^2}\to \mathbb{F}_{q^2}$ be the non-trivial Frobenius automorphism. 
     We suppose $D=P$ is irreducible. By definition, we have
     $$F^{\frac{q^{2\deg P}-1}{3}}\equiv \Omega(\alpha) \moda P,$$
     which is equivalent to say
     $F^{\frac{q^{2\deg P}-1}{3}}=f(T)P+\Omega(\alpha).$ Applying $\sigma$ and using $\sigma(F)=f$, we have  $F^{\frac{q^{2\deg P}-1}{3}}=f^{\sigma}(T)P+\sigma(\Omega(\alpha))\equiv \Omega(\alpha) \moda P$. Note the fact that if two constants are equivalent module $P$, then the two constants are the same. We finally have $\sigma(\Omega(\alpha))= \Omega(\alpha)$. Hence $\Omega(\alpha)=1$. 
     For $D=P_1^{e_1}\cdots P_k^{e_k}$, we have $\chi_D=\chi_{P_1}^{e_1}\cdots \chi_{P_k}^{e_k}$. This completes the proof.
 \end{proof}
 Using the lemma, we obtain
\small{
\begin{align*}
    A_3(s, w)=&\sum_{N\in\Aq}\frac{1}{|N|^{w}}\frac{L_{q^2}(s, \chi^{(N)})}{L_{q^2}(2s, \Bar{\chi}^{(N)})}\sum_{\substack{D\in\Aq\\ (D,N)=1}}\frac{\mu(D)}{|D|^s_2}\prod_{\substack{P_2\in\A2\\ P_2|D}}\left(1+\frac{\chi_{P_2}(N)}{|P_2|^s_2}\right)^{-1}\\
    =&\sum_{N\in\Aq}\frac{1}{|N|^{w}}\frac{L_{q^2}(s, \chi^{(N)})}{L_{q^2}(2s, \Bar{\chi}^{(N)})}\prod_{\substack{P_1\in\Aq\\ (P_1,N)=1}}\left(1-\frac{1}{|P_1|^s_2}\prod_{\substack{P_2|P_1\\
    P_2\in\A2}}\left(1+\frac{\chi_{P_2}(N)}{|P_2|_2^s}\right)^{-1}\right) \\
    =&\sum_{N\in\Aq}\frac{1}{|N|^{w}}\frac{L_{q^2}(s, \chi^{(N)})}{L_{q^2}(2s, \Bar{\chi}^{(N)})}P(s,\chi^{(N)})\prod_{\substack{P_1|N\\ P_1\in\Aq}}\left(1-\frac{1}{|P_1|_2^s}\right)^{-1},
\end{align*}}
where
\begin{equation}
    P(s, \chi^{(N)})=\prod_{P_1\in\Aq}\left(1-\frac{1}{|P_1|_2^s}\prod_{\substack{P_2|P_1\\ P_2\in\A2}}\left(1+\frac{\chi_{P_2}(N)}{|P_2|_2^s}\right)^{-1}\right).
\end{equation}

For $\Re(s)>\frac{1}{2}$, it follows that 
\begin{align}\label{P1}
P(s,\chi^{(N)})=&\prod_{P_1\in\Aq}\left(1-\frac{1}{|P_1|_2^s}+O\left(\frac{1}{|P_1|_2^{3s}}\right)\right)\notag\\
=&\zeta_q^{-1}(2s)\prod\limits_{P\in\Aq}\left(1+O\left(\frac{1}{|P|^{3s}_2}\right)\right).
\end{align}
 
As $(1-\frac{1}{|P_1|_2^s})^{-1}=1+\frac{1}{|P_1|_2^s-1}\ll |P_1|_2^{\max\{0,-\Re(s)\}}$, we will obtain 
\begin{align}\label{P2}
 \prod_{\substack{P_1|N\\ P_1\in\Aq}}\left(1-\frac{1}{|P_1|_2^s}\right)^{-1}\ll |N|_2^{\max\{0,-\Re(s)\}+\varepsilon}.  
\end{align}

Taking $u=q^{-2s}$ and $v=q^{-w}$, we have 
\begin{align}\label{A2}
    A_3(u,v)=\sum_{N\in\Aq}v^{\deg N}\frac{L_{q^2}(u, \chi^{(N)})}{L_{q^2}(u^2, \Bar{\chi}^{(N)})}P(u,\chi^{(N)})\prod_{\substack{P_1|N\\ P_1\in\Aq}}\left(1-u^{\deg P_1}\right)^{-1},
\end{align}
where $$L_{q^2}(u,\chi^{(N)})=\sum\limits_{F\in\A2}\chi_F(N)u^{\deg F}$$ and $$P(u, \chi^{(N)})=\prod\limits_{P_1|\Aq}(1-u^{\deg P_1}\prod\limits_{_{P_2|P_1}}(1+\chi_{P_2}(N)u^{\deg P_2})^{-1}).$$ It is easy to see that $P(u, \chi^{(N)})$ is analytical for $|u|<q^{-1}$.


 When $N$ is not a cube, $\frac{L_{q^2}(u, \chi^{(N)})}{L_{q^2}(u^2, \Bar{\chi}^{(N)})}$ is a rational function in $u$ with possible simple poles on the line $|u|=q^{-\frac{1}{2}}$. The Lindel\"of Hypothesis gives, for $|u|<q^{-1}$ and arbitrary $\varepsilon>0$, \begin{align}\label{L}
 \frac{L_{q^2}(u, \chi^{(N)})}{L_{q^2}(u^2, \Bar{\chi}^{(N)})}\ll |N|_2^{\varepsilon}.\end{align}

There is a possible simple pole at $u=q^{-2}$ when $N$ is a cube in $A_3(u,v)$. In fact, combining \eqref{P1} \eqref{P2},\eqref{A2} and \eqref{L}, for $|u|<q^{-1}$ and any $\varepsilon>0$, we have the bound
\begin{align}
    |(u-q^{-2})A_3(u,v)|\ll |(u-q^{-2})|\sum_{N\in\Aq}|q^{\varepsilon}v|^{\deg N}.
\end{align}
Meanwhile, $(u-q^{-2})A_3(u,v)$ is holomorphic in the convergence region of the right hand side of the above inequality
\begin{align}
    S_1=\{(u,v)||u|<q^{-1}, |v|<q^{-1}\}.
\end{align}

  \subsection{Residue at $u=q^{-2}$}\label{PZ}
  We now analyze the possible poles at $u=q^{-2}$. Let $\zeta_{q^2}(u)=\frac{1}{1-q^2u}$ be the zeta function for field $\mathbb{F}_{q^2}(T)$. The possible pole comes from the term $N$ being a cube. We have
  \begin{align}     
  &\Res_{u=q^{-2}}A_3(u,v)\notag\\
  =&\Res_{u=q^{-2}}\sum_{N\in\Aq}v^{3\deg N}\frac{\zeta_{q^2}(u)}{\zeta_{q^2}(u^2)}\prod_{\substack{P_2\in\A2\\P_2|N}}\frac{1-u^{\deg P_2}}{1-u^{2\deg P_2}}\prod_{P_1 \not| N}(1-u^{\deg P_1}\prod_{\substack{P_2\in\A2\\ P_2|P_1}}(1+u^{\deg P_2})^{-1}).\notag\\
\end{align}
Taking $v=q^{-\frac{1}{2}}$, we have
\begin{align*}
  \Res_{u=q^{-2}}A_3(u,q^{-\frac{1}{2}})=(q^{-4}-q^{-2})P(q^{-2})Z(q^{-2},q^{-\frac{1}{2}}),
  \end{align*}
where 
\begin{equation}
    P(u)=\prod_{P_1\in\Aq}(1-u^{\deg P_1}\prod_{\substack{P_2\in\A2\\ P_2|P_1}}(1+u^{\deg P_2})^{-1})
\end{equation}
and
\begin{equation}
    Z(u,v)=\sum_{N\in\Aq}v^{3\deg N}\prod_{\substack{P_2\in\A2\\P_2|N}}(1+u^{\deg P_2})^{-1}\prod_{\substack{P_1\in\Aq\\ P_1|N}}(1-u^{\deg P_1}\prod_{\substack{P_2\in\A2\\ P_2|P_1}}(1+u^{\deg P_2})^{-1})^{-1}.
\end{equation}

 \subsection{The second region}
The Lindel\"of Hypothesis gives, for $|v|<q^{-\frac{1}{2}}$ and any $\varepsilon$,
$$|L_q(v,\chi_F)|\ll |F|^{\varepsilon}$$
and 
\begin{align*}
   A_3(u,v)=\sum_{\substack{F\in\H2 \\ P|F\Rightarrow P\not\in \mathbb{F}_q[t]}}\frac{L_q(w,\chi_F)}{|F|_2^s}
   \ll \sum_{\substack{F\in\H2 \\ P|F\Rightarrow P\not\in \mathbb{F}_q[t]}}|F|^{\varepsilon}|u|^{\deg F}.
\end{align*}
Then, we obtain convergence region 
\begin{align}
    S_{2,1}=\{(u,v)|~|v|\le q^{-\frac{1}{2}},|u|<q^{-2}\}
\end{align}
and see that $A_3(u,v)$ is holomorphic in it.

Similarly, for $|v|<q^{-\frac{1}{2}}$ and any $\varepsilon>0$, the Lindel\"of Hypothesis and the functional equation give
\begin{align*}
    |L_q(v,\chi_F)|\ll |v^2q^{-1+\varepsilon}|^{\deg F}
\end{align*}
and
\begin{align*}
    A_3(s,w)=q^{2w-1}\frac{1-q^{-w}}{1-q^{w-1}}\sum_{\substack{F\in\mathcal{H}_{q^2} \\ P|F\Rightarrow P\not\in \mathbb{F}_q[t]}}\frac{L_q(1-w, \overline{\chi}_F)}{|F|_2^{s+w-1}}\epsilon(\chi_F).
\end{align*}
Then we have
\begin{align*}
    A_3(u,v)\ll |q^{2w-1}\frac{1-q^{-w}}{1-q^{w-1}}| \sum_{\substack{F\in\H2 \\ P|F\Rightarrow P\not\in \mathbb{F}_q[t]}} |v^2uq^{-1+\varepsilon}|^{\deg F}
\end{align*}
in the convergence region
\begin{align}
    S_{2,2}=\{(u,v)|~|v|>q^{-1},|u^{1/2}v|<q^{-\frac{3}{2}}\}.
\end{align}
 Combining these results, the second overall convergence region is
 \begin{align}
    S_2=S_{2,1}\cup S_{2,2}=\{(u,v)|~|u^{1/2}v|<q^{-\frac{3}{2}}, |u|<q^{-2}\}.
 \end{align}


  \section{The dual term}
  
  \subsection{The third convergence region}
  The functional equation (Lemma~\ref{fe}) gives
  \begin{align}
   A_3(s,w)=&\sum_{\substack{F\in\mathcal{H}_{q^2} \\ P|F\Rightarrow P\not\in \mathbb{F}_q[t]}} \frac{L_q(w, \chi_{F})}{|F|^s_2}\notag\\
   =&q^{2w-1}\frac{1-q^{-w}}{1-q^{w-1}}\sum_{\substack{F\in\mathcal{H}_{q^2} \\ P|F\Rightarrow P\not\in \mathbb{F}_q[t]}}\frac{L_q(1-w, \overline{\chi}_F)}{|F|_2^{s+w}}G_{q^2}(1,F)\notag\\
   =&q^{2w-1}\frac{1-q^{-w}}{1-q^{w-1}}\sum_{N\in\Aq}\frac{1}{|N|^{1-w}}\sum_{D\in\Aq}\frac{\mu(D)\overline{\chi}_D(N)G_{q^2}(1,D)}{|D|_2^{s+w}}\sum_{\substack{F\in\Hq\\ (F,DN)=1}}\frac{G_{q^2}(ND,F)}{|F|_2^{s+w}}.
  \end{align}
By Lemma~\ref{Gauss1}, $G_{q^2}(ND,F)=0$ if $F$ is not square-free. Thus, we can simplify $A_3(s,w)$ to
\begin{align*}
    A_3(s,w)
    =&q^{2w-1}\frac{1-q^{-w}}{1-q^{w-1}}\sum_{N\in\Aq}\frac{1}{|N|^{1-w}}\sum_{\substack{D\in\Aq\\ (D,N)=1}}\frac{\mu(D)G_{q^2}(N,D)}{|D|_2^{s+w}}\sum_{\substack{F\in\A2\\ (F,DN)=1}}\frac{G_{q^2}(ND,F)}{|F|_2^{s+w}}\\
    =&q^{2w-1}\frac{1-q^{-w}}{1-q^{w-1}}C_3(s,w),
\end{align*}
where
\begin{align}
    C_3(s,w)=\sum_{N\in\Aq}\frac{1}{|N|^{1-w}}\sum_{\substack{D\in\Aq\\ (D,N)=1}}\frac{\mu(D)G_{q^2}(N,D)}{|D|_2^{s+w}}H(ND,s+w)
\end{align}
and
\begin{align}
    H(Q,s)=\sum_{\substack{F\in\A2\\ (F,Q)=1}}\frac{G_{q^2}(Q,F)}{|F|_2^s}.
\end{align}
Taking $u=q^{-2s}$ and $v=q^{-w}$, we obtain
\begin{align}
    A_3(u,v)=\frac{1}{qv^2}\frac{1-v}{1-\frac{1}{qv}} C_3(u,v),
\end{align}
where 
\begin{align}
    C_3(u,v)=\sum_{N\in\Aq}\frac{1}{(qv)^{\deg N}}\sum_{\substack{D\in\Aq\\ (D,N)=1}}\mu(D)G_{q^2}(N,D)H(ND,uv^2)(uv^2)^{\deg D}
\end{align}
and
\begin{align}
    H(Q,u)=\sum_{\substack{F\in\A2\\ (F,Q)=1}}G_{q^2}(Q,F)u^{\deg F}.
\end{align}
Here, again, we will not distinct $C_3(s,w)$, $C_3(u,v)$, $H(q,s)$ and $H(Q,u)$.

Replacing $q$ by $q^2$ in Lemma \ref{upper2}, we have $H(ND,uv^2)\ll |u^{\frac{1}{2}}vq^{\frac{3}{2}+\varepsilon}|^{\deg ND}$ for $q^{-\frac{3}{2}}\le |u^{\frac{1}{2}}v|\le q^{-\frac{1}{2}}$. Meanwhile, we have the bound
\begin{align*}
    A_3(u,v)&\ll \sum_{N\in\Aq}\frac{1}{|qv|^{\deg N}}\sum_{\substack{D\in\Aq\\ (D,N)=1}}|G_{q^2}(N,D)||uv^2|^{\deg D}|u^{\frac{1}{2}}vq^{\frac{3}{2}+\varepsilon}|^{\deg ND}\\
    &\ll \sum_{N\in\Aq}\frac{1}{|qv|^{\deg N}}\sum_{\substack{D\in\Aq\\ (D,N)=1}}|D||uv^2|^{\deg D}|u^{\frac{1}{2}}vq^{\frac{3}{2}+\varepsilon}|^{\deg ND}\\
    &\ll \sum_{N\in\Aq} |u^{\frac{1}{2}}q^{\frac{1}{2}+\varepsilon}|^{\deg N}\sum_{\substack{D\in\Aq\\ (D,N)=1}} |q^{\frac{5}{2}+\varepsilon}u^{\frac{3}{2}}v^3|^{\deg D}
\end{align*} 
Thus, $A_3(u,v)$ is holomorphic in region 
$$S_3=\{|u|<q^{-3},|u^{\frac{1}{2}}v|<q^{-\frac{7}{6}}, |v|>1\}.$$
By computing the convex hull of $S_1$, $S_2$ and $S_3$, we extend $(u-q^{-2})A_3(u,v)$ to
\begin{align}
S_4=\{(u,v)||u|<q^{-1}, |u^{\frac{1}{2}}v|<q^{-\frac{7}{6}}, |u^2v^3|<q^{-5}\}.
\end{align}

\begin{remark}
    The poles $(uv^2)=q^{-8}$ are not in $S_4$.
\end{remark}
\subsection{Complete the proof}

 We begin by considering the case where $\Re(s)$ and $\Re(w)$ are sufficiently large, corresponding to $|u|$ and $|v|$ being small. The simple pole $u=q^{-2}$ is contained within $S_4$. The shape of $S_4$ allows us to shift the contour from $|u|=r$ to $|u|=q^{-\frac{7}{8}+\varepsilon}$, while setting $v=q^{-\frac{1}{2}}$. The residue at $u=q^{-2}$ gives the main term of \eqref{main}. For the integration along the contour $|u|=q^{-\frac{7}{8}+\varepsilon}$, we apply a trivial bound, which leads to an additional error term of size $q^{(\frac{7}{8}+\varepsilon)g}$.

  \section*{Acknowledgement}
The authors would like to express their sincere gratitude to Professor Peng Gao for proposing the subject of this study and for providing invaluable guidance and insightful recommendations throughout the research process.


\bibliographystyle{plain}
\bibliography{ref}

\end{document}